\documentclass[12pt]{amsart}
\usepackage[all]{xy}
\usepackage[parfill]{parskip}

\usepackage{verbatim}
\usepackage{color}

\usepackage{amsmath, amscd, graphicx, latexsym, hyperref, times}
\usepackage{color,graphicx}
\usepackage[abs]{overpic}
\usepackage{tikz}
\usepackage{tikz-cd}
\usetikzlibrary{arrows, patterns}

\oddsidemargin .25in \theoremstyle{plain}
\newtheorem{theorem}{Theorem}

\newtheorem{lemma}{Lemma}

\newtheorem{conjecture}{Conjecture}

\newtheorem{remark}{Remark}

\numberwithin{equation}{section}
\numberwithin{lemma}{section}
\numberwithin{proposition}{section}
\numberwithin{corollary}{section}
\numberwithin{remark}{section}
\allowdisplaybreaks[3]

\newcommand{\dd}{\mathop{}\!\mathrm{d}}
\usepackage{cleveref}
\crefname{theorem}{Theorem}{Theorem}
\crefname{lemma}{Lemma}{Lemma}
\crefname{proposition}{Proposition}{Proposition}
\crefname{conjecture}{Conjecture}{Conjecture}
\crefname{equation}{}{}
\crefname{section}{Section}{Section}
\crefname{figure}{Figure}{Figure}
\usepackage{appendix}
\begin{document}

\title[Uniqueness results on manifolds with boundary]{Uniqueness results for positive harmonic functions on manifolds with nonnegative Ricci curvature and strictly convex boundary}
\author{Xiaohan Cai}
\address{School of Mathematical Sciences, Shanghai Jiao Tong University, Shanghai 200240, China}
\email{xiaohancai@sjtu.edu.cn}

\thanks{}
\date{}

\begin{abstract}
We prove some Liouville-type theorems for positive harmonic functions on compact Riemannian manifolds with nonnegative Ricci curvature and strictly convex boundary, thereby confirming some cases of Wang's conjecture (J. Geom. Anal. \textbf{31}, 2021). 

We further investigate Wang's conjecture on warped product manifolds and provide a partial verification of this conjecture,
which also yields an alternative proof of Gu-Li's resolution of the conjecture in the $\mathbb{B}^n$ case (Math. Ann. \textbf{391}, 2025). Our approach is based on a general principle of employing the P-function method to such Liouville-type results, with particular emphasis on the role of a closed conformal vector field inherent to such manifolds. 
\vspace{1em}

\end{abstract}
\maketitle
\section{Introduction}
Motivated by Bidaut-V\'eron and V\'eron's rigidity work \cite[Theorem 6.1]{BVV91} (see also \cite{Ili96}), Wang proposed a conjecture for a Liouville-type result for harmonic functions with some nonlinear boundary condition \cite[Conjecture 1]{Wan21}.  
\begin{conjecture}[Wang, \cite{Wan21}]\label{conj. Wang}
    Let $(M^n,g)\ (n\geq 3)$ be a compact Riemannian manifold with $\mathrm{Ric}\geq 0$ and the second fundamental form $\Pi\geq 1$ on $\partial M$. If $u\in C^{\infty}(M)$ is a positive solution of the following equation
    \begin{align}\label{eq. equation of u on M}
        \begin{cases}
            \Delta u=0  &\text{in }M^n,\\
            \frac{\partial u}{\partial \nu}+\lambda u=u^q &\text{on }\partial M^n,
        \end{cases}
    \end{align}
    where $1<q\leq \frac{n}{n-2}$ and $0<\lambda\leq \frac{1}{q-1}$ are constants. Then either $u$ is a constant function, or $q=\frac{n}{n-2},\, \lambda=\frac{n-2}{2}$, $(M^n,g)$ is isometric to the unit ball $\mathbb{B}^n\subset\mathbb{R}^n$ and $u$ corresponds to
    \begin{align}\label{eq. model solution}
            u(x)=\left(
            \frac{n-2}{2}\frac{1-|a|^2}{|a|^2|x|^2-2\langle x,a\rangle+1}
            \right)^{\frac{n-2}{2}},
        \end{align}
        where $a\in \mathbb{B}^n$.
\end{conjecture}
This conjecture, if proved to be true, has several interesting geometric consequences, such as a sharp Sobolev trace inequality, a sharp upper bound of the area of the boundary, and a sharp lower bound of the Steklov eigenvalue on such manifolds. Readers are invited to \cite[Section 2]{Wan21} and also \cite[Section 5]{GHW21} for expositions of these implications and the connection with the type II Yamabe problem.

Inspired by Xia-Xiong's work on the Steklov eigenvalue estimate \cite{XX24}, Guo-Hang-Wang \cite[Theorem 2]{GHW21} used a remarkable weight function (see \cref{eq. GHW weight function}) to verify \cref{conj. Wang} within the conjectured \emph{sharp} range for $\lambda$ under nonnegative sectional curvature condition:
\begin{theorem}[Guo-Hang-Wang, \cite{GHW21}]\label{thm. GHW21}
    Let $(M^n,g)\ (n\geq 3)$ be a compact Riemannian manifold with $\mathrm{Sec}\geq 0$ and $\Pi\geq 1$ on $\partial M$. Then the only positive solution to \cref{eq. equation of u on M} is constant provided $3\leq n\leq 8$, 
    $1<q\leq \frac{4n}{5n-9}$ and $0<\lambda\leq \frac{1}{q-1}$.
\end{theorem}

We also mention that the two dimensional analogue of \cref{conj. Wang} has been resolved in \cite[Theorem 2]{Wan17} based on a pointwise analysis using the maximum principle, which seemingly fails to work when $n\geq 3$.

In the first part of this paper, we shall partially confirm \cref{conj. Wang} under the Ricci curvature condition via integration by parts and a meticulous choice of the power of the test function.
\begin{theorem}\label{thm. C. Ricci}
    Let $(M^n,g)\ (n\geq 3)$ be a compact Riemannian manifold with $\mathrm{Ric}\geq 0$ and $\Pi\geq 1$ on $\partial M$. Then the only positive solution to \cref{eq. equation of u on M} is constant, provided one of the following two conditions holds:
    \begin{enumerate}
        \item $3\leq n\leq 7$, $1<q\leq \frac{3n}{4(n-2)}$, and $0<\lambda\leq \min\{\frac{1}{2(q-1)}, \frac{3(n-1)}{2q}\}$\\
        \item $3\leq n\leq 9$, $1<q\leq \frac{n+1+\sqrt{(5n-1)(n-1)}}{4(n-2)}$
        and $0<\lambda\leq \min\{\frac{1}{2(q-1)},\frac{6q+1}{2q+1}\frac{n-1}{2q}\}$
    \end{enumerate}
\end{theorem}
\begin{remark}
    Let us compare the ranges of $q$ and $\lambda$ in \cref{thm. C. Ricci} with those in \cref{thm. GHW21}.
    \begin{itemize}
    \item Condition (2) in \cref{thm. C. Ricci} admits an extra dimension, $n=9$, despite relying on a weaker curvature assumption.
        \item \cref{thm. C. Ricci} yields a larger range of $q$ than \cref{thm. GHW21} in some dimensions. Explicitly, we have   $\frac{n+1+\sqrt{(5n-1)(n-1)}}{4(n-2)}\geq \frac{3n}{4(n-2)}\geq \frac{4n}{5n-9}$, if  $3\leq n\leq 5$; and $\frac{n+1+\sqrt{(5n-1)(n-1)}}{4(n-2)}
        \geq \frac{4n}{5n-9}
        \geq \frac{3n}{4(n-2)}$, if $6\leq n\leq 9$.
        \item 
        In view of the range of $\lambda$, condition (2) doesn't cover condition (1) since $\frac{6q+1}{2q+1}\frac{n-1}{2 q}<\frac{3(n-1)}{2q}$.
    \end{itemize}
\end{remark}
\begin{remark}
    After completing this work, the author learned of the work by Shouman \cite{Sho26}, which independently establishes almost the same result as \cref{thm. C. Ricci}. It is worth noting that \cite{Sho26} furthermore derives a logarithmic-Sobolev type inequality as a corollary of this Liouville type theorem.
\end{remark}

	Except for studying it for general manifolds within some ranges of the parameters,  another way toward \cref{conj. Wang} is to confine ourselves to some specific manifolds. Guo and Wang  proposed such an individual conjecture on the model space $\mathbb{B}^n$  \cite[Conjecture 1]{GW20}:
    \begin{conjecture}[Guo-Wang, \cite{GW20}]\label{conj.Guo-Wang}
        If $u\in C^{\infty}(\mathbb{B}^n)$ is a positive solution of the  equation
        \begin{equation}\label{eq. harmonic}
            \begin{cases}
                \Delta u=0 &\mathrm{ in} \ \mathbb{B}^n,\\
                \frac{\partial u}{\partial\nu}+\lambda u=u^q &\mathrm{on}\ \mathbb{S}^{n-1},
            \end{cases}
        \end{equation}
        where $1<q\leq\frac{n}{n-2}$ and $0<\lambda \leq \frac{1}{q-1}$ are constants. Then either $u$ is a constant function, or $q=\frac{n}{n-2},\, \lambda=\frac{n-2}{2}$ and $u$ corresponds to \cref{eq. model solution}:
    \begin{align*}
            u(x)=\left(
            \frac{n-2}{2}\frac{1-|a|^2}{|a|^2|x|^2-2\langle x,a\rangle+1}
            \right)^{\frac{n-2}{2}},
        \end{align*}
        for some $a\in \mathbb{B}^n$.
    \end{conjecture}

    Historically, Escobar \cite[Theorem 2.1]{Esc90}  classified all solution of \cref{eq. harmonic} by Obata's integral method \cite{Oba71} when $q=\frac{n}{n-2}$ and $\lambda=\frac{n-2}{2}$ (see also \cite{LZ95} for another proof using method of moving spheres). After several contributions  in this direction \cite{GW20,GHW21,LO23,Ou24}, Gu-Li \cite[Theorem 1.1]{GL25} finally give an affirmative answer to \cref{conj.Guo-Wang}.
    \begin{theorem}[Gu-Li, \cite{GL25}]\label{thm. GL25}
        If $u\in C^{\infty}(\mathbb{B}^n)$ is a positive solution of the  equation
        \cref{eq. harmonic}
        for some constants $1<q<\frac{n}{n-2}$ and $0<\lambda \leq \frac{1}{q-1}$, then $u$ is a constant function.
    \end{theorem}
    Gu-Li's method is based on sophisticated integration by parts with several delicately chosen parameters.

As a further step upon the resolution of Guo-Wang's \cref{conj.Guo-Wang}, we investigate the  case that the manifold $(M^n,g)$ is a warped product over $\mathbb{S}^{n-1}$.

\begin{theorem}\label{thm. C. warped product}
    Let $(M^n,g)
        =([0,R]\times \mathbb{S}^{n-1},dr^2+\rho(r)^2dS_{n-1}^2)$ be a smooth warped product manifold (i.e. $\rho(0)=0,\rho'(0)=1, \rho^{(\text{even})}=0$). Assume $Ric_g\geq 0$ and the second fundamental form $\Pi\geq 1$ (i.e. $\frac{\rho'(R)}{\rho(R)}\geq 1$). Assume, in addition, that $\rho'(R)>\frac{n-2}{n-1}$.

        Consider $1<q\leq \frac{2(n-1)\rho'(R)-(n-2)}{n-2}$ and  $0<\lambda\leq
        \frac{1}{q-1}C(n,q,\rho'(R))$, where $C(n,q,\rho'(R))\in(0,1]$ is a constant depending on $n,q,\rho'(R)$ (see \cref{def. C} for its explicit definition).
        Then the positive solution $u$ to the equation \cref{eq. equation of u on M} must be constant, unless $(M^n,g)$ is isometric to $\mathbb{B}^n$, $q=\frac{n}{n-2}$, $\lambda=\frac{n-2}{2}$ and u corresponds to the model solution \cref{eq. model solution}.
\end{theorem}
\begin{remark}
    One could see that $C(n,q,1)\equiv 1, \forall n\geq 3, 1<q\leq \frac{n}{n-2}$, so \cref{thm. C. warped product} recovers \cref{thm. GL25} by taking $(M^n,g)$ to be the model space $\mathbb{B}^n$ (i.e. $\rho(r)=r,\forall r\in[0,1]$).
\end{remark}
\begin{remark}
    \cref{thm. C. warped product} applies to a wide range of warped product manifolds. For instance,
    consider the warping factor $\rho(r)=r-cr^3,\ r\in [0,R]$, where $c\geq 0$ is a constant. It's easy to see that if $c,R$ satisfy
    \begin{align*}
    1-R\geq cR^2(3-R)
    ,\quad \text{and}\quad 
        cR^2<\frac{1}{3(n-1)},
    \end{align*}
    then the hypotheses in \cref{thm. C. warped product} are fulfilled.
\end{remark}

Our approach to \cref{thm. C. warped product} is grounded in a general principle for applying the P-function method to classify solutions of such semilinear elliptic equations, thereby building upon and extending the idea presented in \cite{Wan22}. 
The basic idea is to start with the critical power case (i.e. $q=\frac{n}{n-2}$) and study the model solution \cref{eq. model solution} and come up with an appropriate function, known as the P-function in literature (in honor of L. Payne), whose constancy implies the rigidity of the solution $u$. For the subcritical power case (i.e. $1<q<\frac{n}{n-2}$), we regard the equation \cref{eq. equation of u on M} as the critical one in a $d-$dimensional space (see \cref{eq. intrinsic dimension}), then a modified argument as in the critical power case would  imply the conclusion. Readers interested in the P-function method could refer to \cite{Pay68,Wei71,Dan11,CFP24} for more research in  this realm.

Specific to \cref{thm. C. warped product}, we start with a divergence-type equation of the P-function (see \cref{eq. div 1_1-d}) to obtain a key integral identity. A new weight function (see \cref{eq. weight in warped case}) is then employed to handle the challenging boundary terms and derive the rigidity of the solution $u$.
Our approach highlights the role of  this weight function as offering a closed conformal vector field (see \cref{cor. div 1w} and \cref{eq. warped 1w_1-d}) and leveraging the warped product structure of the underlying manifolds. 
It's worth mentioning that our arguments also yield an alternative proof of \cref{thm. GL25} and elucidate the choice of parameters in their proof from the viewpoint of this P-function method.

Note that for a warped product manifold with nonnegative Ricci curvature, its radial-direction sectional curvature is also nonnegative, from which the arguments of Guo-Hang-Wang \cite{GHW21} would follow and yield a \emph{sharp} range of $\lambda$ for $3\leq n\leq 8$. In contrast, while potentially sacrificing sharpness in $\lambda$, we obtain a valid range for \emph{all} dimensions $n$. We shall discuss this difference in detail after the proof of \cref{thm. C. warped product}.

Finally, we notice that if we confine ourselves to classifying the minimizer of the  Sobolev trace inequality on warped product manifolds, then a larger range of $\lambda$ would be admitted.
\begin{theorem}\label{thm.C. minimizer}
        Let $(M^n,g)
        =([0,R]\times \mathbb{S}^{n-1},dr^2+\rho(r)^2dS_{n-1}^2)$ be a smooth warped product manifold (i.e. $\rho(0)=0,\rho'(0)=1, \rho^{(\text{even})}=0$). Assume $Ric_g\geq 0$ and the second fundamental form $\Pi\geq 1$ (i.e. $\frac{\rho'(R)}{\rho(R)}\geq 1$). Assume, in addition, that $\rho'(R)>\frac{n-2}{n-1}$.
        
         Consider $1<q\leq \frac{2(n-1)\rho'(R)-(n-2)}{n-2}$ and $0<\lambda\leq 
         \min\{\frac{n-1}{q+1}\frac{\rho'(R)}{\rho(R)}, \frac{1}{q-1}\frac{\rho(R)}{\rho'(R)} \}$ . Let $u_0$ be a minimizer of the Sobolev trace functional 
           \begin{align*}
               Q_{q,\lambda}(u):=\frac{\int_{M}
               |\nabla u|^2
               +\lambda \int_{\partial M} u^2 }{(\int_{\partial M}u^{q+1})^{\frac{2}{q+1}}},\quad 
               \forall u\in C^{\infty}(M)\setminus C_0^{\infty}(M).
           \end{align*}
           Then $u_0$ is constant, 
           unless $(M^n,g)$ is isometric to $\mathbb{B}^n$, 
           $q=\frac{n}{n-2}, \lambda=\frac{n-2}{2}$, 
           and $u_0$ is given by the model solution \cref{eq. model solution}
           up to multiplying by a constant.
    \end{theorem}

    The idea toward \cref{thm.C. minimizer}  originates from Escobar \cite[Theorem 2]{Esc88}, in which he classified the minimizer of the critical power case (i.e. $q=\frac{n}{n-2}$) of the Sobolev trace inequality on $\mathbb{B}^n$. 
    As in \cref{thm. C. warped product}, the closed conformal vector field also  plays a vital role in our proof.

    This paper is organized as follows. In \cref{sec2}, we study \cref{conj. Wang} on general manifolds and prove \cref{thm. C. Ricci}. In \cref{sec3}, we confine ourselves to warped product manifolds and give the proof of \cref{thm. C. warped product}, which also provides an alternative proof of \cref{thm. GL25} by taking $\rho(r)=r,\forall r\in[0,1]$. After finishing the proof, we shall have a discussion about the connection and difference between our \cref{thm. C. warped product} and Guo-Hang-Wang's \cref{thm. GHW21}.
    Finally, in \cref{Sec.4}, we classify the minimizers on warped product manifolds and prove \cref{thm.C. minimizer}. 
 
\section{The proof of \cref{thm. C. Ricci}}\label{sec2}
Our proof exploits two main ingredients in the method of integration by parts: Bochner formula (Step 1) and the equation itself (Step 2).
\begin{proof}[Proof of \cref{thm. C. Ricci}:]
\textbf{Step 1:}
    Let $v:=u^{-\frac{1}{a}}$, where $a>0$ is to be determined. We follow the notations in \cite{GHW21} and define $\chi:=\frac{\partial v}{\partial \nu},\, f:=v|_{\partial M}$, where $\nu$ is the unit outer normal vector of $\partial M$. Then it's straightforward to see that
    \begin{align}\label{eq. equation of v on M}
        \begin{cases}
            \Delta v=(a+1)v^{-1}|\nabla v|^2 &\text{in }M,\\
            \chi=\frac{1}{a}(\lambda f-f^{a+1-aq}) &\text{on }\partial M.
        \end{cases}
    \end{align}
    By Bochner formula, there holds
    \begin{align*}
        \frac{1}{2}\Delta |\nabla v|^2
        =\left|\nabla^2 v-\frac{\Delta}{n}g\right|^2
        +\frac{1}{n}(\Delta v)^2
        +\langle \nabla \Delta v,\nabla v\rangle
        +\mathrm{Ric}(\nabla v,\nabla v).
    \end{align*}
    Multiply both sides by $v^b$ and integrate it over $M$, where $b$ is a constant to be determined, we have
    \begin{align}\label{eq. Bochner 1}
        \frac{1}{2}\int_M v^b\Delta |\nabla v|^2
        =\int_Mv^b\left(\left|\nabla^2 v-\frac{\Delta}{n}g\right|^2+\mathrm{Ric}(\nabla v,\nabla v) \right)
        +\frac{1}{n}\int_M v^b(\Delta v)^2
        +\int_M v^b\langle \nabla \Delta v,\nabla v\rangle
        .
    \end{align}
    It follows from \cref{eq. equation of v on M}, divergence theorem and the boundary curvature assumption $\Pi\geq 1, H\geq n-1$ that the left hand side of \cref{eq. Bochner 1} equals
    \begin{align}\label{eq. LHS of 1}
        &\frac{1}{2}\int_M v^b\Delta |\nabla v|^2
        =\frac{1}{2}\int_M \mathrm{div}(v^b\nabla|\nabla v|^2)
        -\langle\nabla v^b,\nabla |\nabla v|^2\rangle\notag\\
        =&\int_{\partial M}f^b\langle\nabla_{\nabla f+\chi\nu}\nabla v,\nu\rangle
        -b\int_Mv^{b-1}\langle\nabla_{\nabla v}\nabla v,\nabla v\rangle\notag\\
        =&\int_{\partial M}f^b\left(
        \langle\nabla f,\nabla \chi\rangle-\langle\nabla v,\nabla_{\nabla f}\nu\rangle+\chi\nabla^2v(\nu,\nu)
        \right)
        -b\int_Mv^{b-1}\langle\nabla_{\nabla v}\nabla v,\nabla v\rangle\notag\\
        =&\int_{\partial M}f^b\left(
        \langle\nabla f,\nabla \chi\rangle
        -\Pi(\nabla f,\nabla f)
        +\chi(\Delta v-\Delta f-H\chi)
        \right)
        -b\int_Mv^{b-1}\langle\nabla_{\nabla v}\nabla v,\nabla v\rangle\notag\\
        =&\int_{\partial M}f^b
        \left(
        2\langle\nabla f,\nabla \chi\rangle
        -\Pi(\nabla f,\nabla f)
        +(a+1)\chi f^{-1}(|\nabla f|^2+\chi^2)
        +bf^{-1}\chi|\nabla f|^2
        -H\chi^2
        \right)\notag\\
        &-b\int_Mv^{b-1}\langle\nabla_{\nabla v}\nabla v,\nabla v\rangle\notag\\
        \leq&(2(q-1)\lambda-1)\int_{\partial M}f^b|\nabla f|^2
        +\left(a+1+b+2(a+1-aq)\right)\int_{\partial M}f^b\frac{\chi}{f}|\nabla f|^2
        +(a+1)\int_{\partial M}f^b\frac{\chi}{f}\chi^2\notag\\
        &-(n-1)\int_{\partial M}f^b\chi^2
        -b\int_M v^{b-1}\langle\nabla_{\nabla v}\nabla v,\nabla v\rangle\notag\\
        =&(2(q-1)\lambda-1)\int_{\partial M}f^b|\nabla f|^2
        +\frac{a+1+b+2(a+1-aq)}{a}\lambda\int_{\partial M}f^b|\nabla f|^2\notag\\
        &-\frac{a+1+b+2(a+1-aq)}{a}\int_{\partial M}f^{b+a-aq}|\nabla f|^2
        +\frac{a+1}{a}\lambda\int_{\partial M}f^b\chi^2
        -\frac{a+1}{a}\int_{\partial M}f^{b+a-aq}\chi^2\notag\\
        &-(n-1)\int_{\partial M}f^b\chi^2
        -b\int_M v^{b-1}\langle\nabla_{\nabla v}\nabla v,\nabla v\rangle\notag\\
        =&\left(\frac{a+b+3}{a}\lambda-1\right)
        \int_{\partial M}f^b|\nabla f|^2
        +\left(\frac{a+1}{a}\lambda-(n-1)\right)
        \int_{\partial M}f^b\chi^2\notag\\
        &+\left(2(q-1)-\frac{a+b+3}{a}\right)
        \int_{\partial M}f^{b+a-aq}|\nabla f|^2
        -\frac{a+1}{a}\int_{\partial M}f^{b+a-aq}\chi^2
        -b\int_M v^{b-1}\langle\nabla_{\nabla v}\nabla v,\nabla v\rangle.
    \end{align}
    On the other hand, the right hand side of \cref{eq. Bochner 1} could be written as
    \begin{align}\label{eq. RHS of 1}
         &\int_M v^b\left(
        \left|\nabla^2 v-\frac{\Delta v}{n}g\right|^2
        +\mathrm{Ric}(\nabla v,\nabla v)
        \right)
        +\frac{(a+1)^2}{n}\int_M w v^{b-2}|\nabla v|^4
        -(a+1)\int_M v^{b-2}|\nabla v|^4\notag\\
        &+2(a+1)\int_M v^{b-1}\langle\nabla_{\nabla v}\nabla v,\nabla v\rangle.
    \end{align}
    Therefore, substituting \cref{eq. Bochner 1} by \cref{eq. LHS of 1} and \cref{eq. RHS of 1}, we have
    \begin{align}\label{eq. equation 1}
        &\left(\frac{a+b+3}{a}\lambda-1\right)
        \int_{\partial M}f^b|\nabla f|^2
        +\left(\frac{a+1}{a}\lambda-(n-1)\right)
        \int_{\partial M}f^b\chi^2\notag\\
        &+\left(2(q-1)-\frac{a+b+3}{a}\right)\int_{\partial M}f^{b+a-aq}|\nabla f|^2\notag
        -\frac{a+1}{a}\int_{\partial M}f^{b+a-aq}\chi^2\\
        \geq & \int_M v^b\left(
        \left|\nabla^2 v-\frac{\Delta v}{n}g\right|^2
        +\mathrm{Ric}(\nabla v,\nabla v)
        \right)
        +\left(\frac{(a+1)^2}{n}-(a+1)\right)\int_M v^{b-2}|\nabla v|^4\notag\\
        &+(2(a+1)+b)\int_M v^{b-1}\langle\nabla_{\nabla v}\nabla v,\nabla v\rangle.
    \end{align}
    \textbf{Step 2:} Multiply both sides of \cref{eq. equation of v on M} by $v^b$ and integrate it over $M$, we  have
    \begin{align}\label{eq. "Bochner" 2}
        \int_{M}v^b(\Delta v)^2
        =(a+1)\int_M v^{b-1}|\nabla v|^2\Delta v
        .
    \end{align}
    It follows from \cref{eq. equation of v on M} that the left hand side of \cref{eq. "Bochner" 2} could be written as
    \begin{align}\label{eq. LHS of 2}
        &\int_{M}v^b(\Delta v)^2
        =\int_M \mathrm{div}(v^b\Delta v\nabla v)
        -\langle \nabla v^b,\nabla v\rangle\Delta v
        -\langle\nabla \Delta v,\nabla v\rangle v^b\notag\\
        =&(a+1)\int_{\partial M}f^{b-1}(|\nabla f|^2+\chi^2)\chi
        -(a+1)b\int_Mv^{b-2}|\nabla v|^4\notag\\
       & +(a+1)\int_M v^{b-2}|\nabla v|^4
        -2(a+1)\int_M v^{b-1}\langle\nabla_{\nabla v}\nabla v,\nabla v\rangle\notag\\
        =&\frac{a+1}{a}\lambda\int_{\partial M}f^b|\nabla f|^2
        +\frac{a+1}{a}\lambda\int_{\partial M}f^b\chi^2
        -\frac{a+1}{a}\int_{\partial M}f^{b+a-aq}|\nabla f|^2
        -\frac{a+1}{a}\int_{\partial M}f^{b+a-aq}\chi^2\notag\\
        &+(a+1)(1-b)\int_M v^{b-2}|\nabla v|^4
        -2(a+1)\int_M v^{b-1}\langle\nabla_{\nabla v}\nabla v,\nabla v\rangle.
    \end{align}
    Therefore, inserting \cref{eq. LHS of 2} into \cref{eq. "Bochner" 2},  we have
    \begin{align}\label{eq. equation 2}
        &\frac{a+1}{a}\lambda\int_{\partial M}f^b|\nabla f|^2
        +\frac{a+1}{a}\lambda\int_{\partial M}f^b\chi^2
        -\frac{a+1}{a}\int_{\partial M}f^{b+a-aq}|\nabla f|^2
        -\frac{a+1}{a}\int_{\partial M}f^{b+a-aq}\chi^2\notag\\
        =&\left(
        (a+1)^2+(a+1)(b-1)
        \right)\int_M v^{b-2}|\nabla v|^4
        +2(a+1)\int_M v^{b-1}\langle\nabla_{\nabla v}\nabla v,\nabla v\rangle.
    \end{align}
    \textbf{Step 3:} Now consider $\cref{eq. equation 1}+c\cref{eq. equation 2}$:
    \begin{align*}
        &\left(\frac{a+b+3}{a}\lambda-1+\frac{a+1}{a}\lambda c\right)
        \int_{\partial M}f^b|\nabla f|^2
        +\left(\frac{a+1}{a}\lambda-(n-1)+\frac{a+1}{a}\lambda c\right)
        \int_{\partial M}f^b\chi^2\notag\\
        &+\left(2(q-1)-\frac{a+b+3}{a}-\frac{a+1}{a}c\right)
        \int_{\partial M}f^{b+a-aq}|\nabla f|^2\notag
        -(\frac{a+1}{a}+\frac{a+1}{a}c)\int_{\partial M}f^{b+a-aq}\chi^2\\
        \geq & \int_M v^b\left(
        \left|\nabla^2 v-\frac{\Delta v}{n}g\right|^2
        +\mathrm{Ric}(\nabla v,\nabla v)
        \right)
        +\left((\frac{1}{n}+c)(a+1)^2
        +(cb-c-1)(a+1)\right)\int_M v^{b-2}|\nabla v|^4\notag\\
        &+\left(2(c+1)(a+1)+b\right)\int_M v^{b-1}\langle\nabla_{\nabla v}\nabla v,\nabla v\rangle.
    \end{align*}
   Define $\beta$ by setting $b=-\beta(a+1)$, and $x:=\frac{1}{a}$ and choose $c=-1+\frac{\beta}{2}$ to eliminate the  term $\int_M v^{b-1}\langle \nabla_{\nabla v}\nabla v,\nabla v\rangle$, we obtain
    \begin{align}\label{eq. Final ineq.}
        &\left(\frac{4-\beta}{2}\lambda x-(1+\frac{\beta\lambda}{2})\right)
        \int_{\partial M}f^b|\nabla f|^2
        +\left(\frac{\beta\lambda}{2}(x+1)-(n-1)\right)
        \int_{\partial M}f^b\chi^2\notag\\
        &+\left(-\frac{4-\beta}{2}x+2(q-1)+\frac{1}{2}\beta\right)
        \int_{\partial M}f^{b+a-aq}|\nabla f|^2\notag
        -\frac{1}{2}\beta(x+1)\int_{\partial M}f^{b+a-aq}\chi^2\\
        \geq & \int_M v^b\left(
        \left|\nabla^2 v-\frac{\Delta v}{n}g\right|^2
        +\mathrm{Ric}(\nabla v,\nabla v)
        \right)\notag\\
        &+\frac{1}{x^2}(x+1)\left(
        -(\frac{1}{2}\beta^2-\beta+\frac{n-1}{n})(x+1)+\frac{1}{2}\beta
        \right)
        \int_M v^{b-2}|\nabla v|^4.
    \end{align}
    \textbf{Step 4:} Now we verify the condition (1) in \cref{thm. C. Ricci}: Take $\beta=1, x=\frac{\beta+4(q-1)}{4-\beta}=\frac{1+4(q-1)}{3}$ such that $-\frac{4-\beta}{2}x+2(q-1)+\frac{1}{2}\beta=0$. Equivalently, $a=\frac{3}{1+4(q-1)}$ and $b=-(a+1)=-\frac{4q}{4q-3}$. Then \cref{eq. Final ineq.} turns out to be
    \begin{align}\label{eq. Final (a)}
        &\left(2(q-1)\lambda-1 \right)\int_{\partial M}f^b|\nabla f|^2
        +\left(\frac{2}{3}q\lambda-(n-1)\right)\int_{\partial M}f^b\chi^2
        -\frac{2}{3}q\int_{\partial M}f^{b+a-aq}\chi^2\notag\\
        \geq& \int_M v^b\left(
        \left|\nabla^2 v-\frac{\Delta v}{n}g\right|^2
        +\mathrm{Ric}(\nabla v,\nabla v)
        \right)
        +\left(\frac{3}{1+4(q-1)}\right)^2
        \frac{4}{3}q
        \left(\frac{1}{2}-\frac{2(n-2)}{3n}q\right)
        \int_M v^{b-2}|\nabla v|^4.
    \end{align}
    Then the condition (1) in \cref{thm. C. Ricci} implies that the left hand side of \cref{eq. Final (a)} is less or equal to zero, while the right hand side of \cref{eq. Final (a)} is larger or equal to zero. It follows that $0\equiv\chi=\frac{\partial v}{\partial \nu}$. Hence by \cref{eq. equation of v on M} we deduce that $v|_{\partial M}=f$ is constant, and $u$ is a harmonic function in $M$ with constant boundary value on $\partial M$. Therefore $u$ is constant in $M$.

    \textbf{Step 5:} Now we verify the condition (2) in \cref{thm. C. Ricci}: By analyzing the range of $x$ within which the left hand side of \cref{eq. Final ineq.} is less or equal to zero and the right hand side of \cref{eq. Final ineq.} is greater or equal to zero, we obtain the optimal value of $\beta$ as $\beta=\frac{4q+2}{4q+1}$, $x=\frac{\beta+4(q-1)}{4-\beta}=\frac{8q^2-4q-1}{6q+1}$. Equivalently, we choose $a=\frac{6q+1}{8q^2-4q-1}$ and $b=-\frac{4q+2}{4q+1}(a+1)=-\frac{4q(2q+1)}{8q^2-4q-1}$. Then \cref{eq. Final ineq.} turns out to be
    \begin{align}\label{eq. Final (b)}
        &\left(2(q-1)\lambda-1\right)
        \int_{\partial M}f^b|\nabla f|^2
        +\left(\frac{2q(2q+1)}{6q+1}\lambda-(n-1)\right)
        \int_{\partial M}f^b\chi^2
        -\frac{2q(2q+1)}{6q+1}\int_{\partial M}f^{b+a-aq}\chi^2\notag\\
        \geq & \int_M v^b\left(
        \left|\nabla^2 v-\frac{\Delta v}{n}g\right|^2
        +\mathrm{Ric}(\nabla v,\nabla v)
        \right)\notag\\
        &+\left(\frac{6q+1}{8q^2-4q-1}\right)^2
        \left(\frac{2q(4q+1)}{6q+1}\right)
        \left(
        \frac{-4(n-2)q^2+2(n+1)q+n}{(6q+1)n}
        \right)
        \int_M v^{b-2}|\nabla v|^4.
    \end{align}
    Then the condition (2) in \cref{thm. C. Ricci} implies that the left hand side of \cref{eq. Final (b)} is less or equal to zero, while the right hand side of \cref{eq. Final (b)} is larger or equal to zero. As before, we could conclude that $u$ is a constant function on $M$.
\end{proof}
 \section{Wang's conjecture on warped product manifolds}\label{sec3}
 In this section, we investigate \cref{conj. Wang} on warped product manifolds and prove \cref{thm. C. warped product}.
 
 We shall first establish some identities that hold on general Riemannian manifolds.
\begin{lemma}\label{cor.div 1}
    Let $(M^n,g)$ be a Riemannian manifold. For any $v\in C^{\infty}(M)$ and constant $d>0$, there holds
    \begin{align}\label{eq. div 1}
        \mathrm{div}\left(\nabla_{\nabla v}\nabla v-\frac{\Delta v}{d}\nabla v\right)
        =\left|\nabla ^2 v-\frac{\Delta v}{n}g\right|^2
        +\left(\frac{1}{n}-\frac{1}{d}\right)(\Delta v)^2
        +\left(1-\frac{1}{d}\right)
        \langle\nabla \Delta v,\nabla v\rangle+\mathrm{Ric}(\nabla v,\nabla v).
    \end{align}
\end{lemma}
\begin{proof}
    This is a straightforward corollary of the Bochner formula.
\end{proof}
\begin{lemma}\label{cor. div 1w}
    Let $(M^n,g)$ be a Riemannian manifold. Assume in addition that $(M^n,g)$ admits a smooth function $w$ such that $\nabla w$ is a closed conformal vector filed on $M$, i.e. $\nabla ^2 w=\frac{\Delta w}{n}g$, then for any $v\in C^{\infty}(M)$ and constant $d>0$ there holds
    \begin{align}\label{eq. div 1w}
        \mathrm{div}\left(\nabla_{\nabla w}\nabla v-\frac{\Delta v}{d}\nabla w\right)
        =\left(\frac{1}{n}-\frac{1}{d}\right)\Delta v\Delta w
        -\left(1-\frac{1}{n}\right)\langle\nabla \Delta w,\nabla v\rangle
        +\left(1-\frac{1}{d}\right)\langle\nabla \Delta v,\nabla w\rangle.
    \end{align}
\end{lemma}
\begin{proof}
    Notice that
    \begin{align*}
        \mathrm{div}(\nabla_{\nabla w}\nabla v)
        &=\frac{1}{2}\langle \nabla^2 v,\mathcal{L}_{\nabla w}g\rangle
        +\langle\nabla\Delta v,\nabla w\rangle
        +Ric(\nabla v,\nabla w)\\
        &=\frac{1}{n}\Delta v\Delta w
        +\langle\nabla\Delta v,\nabla w\rangle
        +Ric(\nabla v,\nabla w),
    \end{align*}
    where we used Ricci's identity
    \begin{align*}
        \mathrm{div}(\nabla^2 v)=\nabla \Delta v+\mathrm{Ric}(\nabla v,\cdot).
    \end{align*}
    The hypothesis on $w$ implies that the curvature operator $R$ satisfies
    \begin{align*}
        R(X,Y)\nabla w=\left\langle X, \frac{\nabla \Delta w}{n}\right\rangle Y
        -\left\langle Y, \frac{\nabla \Delta w}{n}\right\rangle X,\quad \forall X,Y\in T_p(M).
    \end{align*}
    It follows that 
    \begin{align*}
        Ric(\nabla w,\cdot)=-\frac{n-1}{n}\nabla \Delta w.
    \end{align*}
    Therefore, we derive
    \begin{align*}
        \mathrm{div}(\nabla_{\nabla w}\nabla v)
        =\frac{1}{n}\Delta v\Delta w
        +\langle\nabla\Delta v,\nabla w\rangle
        -\frac{n-1}{n}\langle\nabla \Delta w,\nabla v\rangle.
    \end{align*}
    Hence the desired identity follows.
\end{proof}
\begin{remark}
    Manifolds endowed with a closed conformal vector field could be characterized as certain warped product manifolds. See for example \cite[Theorem 1.1]{CMM12}.
\end{remark}

Now we are ready to prove \cref{thm. C. warped product}.
Roughly speaking, the starting observation is that Escobar's argument in the critical power case \cite{Esc90} could be interpreted in the framework of the P-function method. Then we regard the equation \cref{eq. equation of u on M} as the  critical power case in a $d-$dimensional space, where \begin{align}\label{eq. intrinsic dimension}
    d:=\frac{2q}{q-1}\geq n.
\end{align}
 Then we modify Escobar's argument \cite{Esc90} to fit in this "critical case", and use the boundary condition to tackle the emerging terms in this case. We shall see that the choice of the parameters is natural from the viewpoint of this P-function strategy. After finishing our proof, we shall review and compare our choice of the parameters with those of Gu-Li \cite{GL25}.

\begin{proof}[Proof of \cref{thm. C. warped product}:] 

\textbf{Step 0:} First, we derive some properties of the warping factor $\rho$ that will be used throughout our proof.

The curvature hypothesis $Ric_g\geq 0$ implies 
        $-(n-1)\frac{\rho''}{\rho}\geq 0$.
        Combining with $\rho'(0)=1$, we get 
        \begin{align}\label{ineq. rho prime r}
            \rho'(R)
            \leq \rho'(r)
            \leq  1,\quad 
            \forall r\in[0,R].
        \end{align}
        Moreover, the Ricci curvature hypothesis also implies
        \begin{align*}
        (\frac{\rho'}{\rho})'=\frac{\rho''}{\rho}
        -(\frac{\rho'}{\rho})^2
        \leq -(\frac{\rho'}{\rho})^2,\quad 
    \end{align*}
   Combining with $\lim_{r\to 0}\frac{\rho'(r)}{\rho(r)}=+\infty$, we obtain
    \begin{align*}
        \frac{\rho'(r)}{\rho(r)}\leq \frac{1}{r},
        \quad 
        \forall r\in(0,R].
    \end{align*}
    Then the boundary curvature hypothesis implies that 
    \begin{align}\label{ineq. R leq 1}
        R\leq 1.
    \end{align}

    Finally, we have $\rho'(r)\geq \rho'(R)>0$, so 
    \begin{align}\label{ineq. rho r leq rho R}
        \rho(r)\leq \rho(R),\quad 
        \forall r\in[0,R].
    \end{align}

\textbf{Step 1:}    Consider the intrinsic dimension $d:=\frac{2q}{q-1}$ so that $q=\frac{d}{d-2}$. Then 
\begin{align*}
    d\geq n\Leftrightarrow q\leq \frac{n}{n-2}.
\end{align*}
Denote  $\nu$ to be the unit outer normal vector of $\partial M$.  Let  $v:=u^{-\frac{2}{d-2}}$, $\chi:=\frac{\partial v}{\partial \nu}, f:=v|_{\partial M}$, then $v$ satisfies
    \begin{align}\label{eq. PDE of v on warped}
        \begin{cases}
            \Delta v=\frac{d}{2}v^{-1}|\nabla v|^2 &\mathrm{in}\ M,\\
            \chi=\frac{2}{d-2}(\lambda f-1) &\mathrm{on}\ \partial M.
        \end{cases}
    \end{align}
    Consider the P-function $P:=v^{-1}|\nabla v|^2=\frac{2}{d}\Delta v$, which satisfies $P\equiv constant$ if $u$ is the model solution \cref{eq. model solution}. We shall derive a divergence-type equation of $P$ (i.e. \cref{eq. div 1_1-d}).

    It's straightforward to see from \cref{eq. PDE of v on warped} that
    \begin{align}\label{eq. nabla Delta v}
        \frac{d}{2}\nabla P=\nabla \Delta v
        =dv^{-1}
        \left(\nabla_{\nabla v}\nabla v-\frac{\Delta v}{d}\nabla v\right).
    \end{align}

    Applying the Bochner formula to $vP=|\nabla v|^2$, it's straightforward to see that
    \begin{align*}
        \frac{1}{2}v\Delta P +\frac{1}{2}P\Delta v+\langle\nabla P,\nabla v\rangle
        =\left|\nabla^2 v-\frac{\Delta v}{n}g\right|^2
        +\frac{1}{n}(\Delta v)^2
        +\langle\nabla \Delta v,\nabla v\rangle
        +Ric(\nabla v,\nabla v)
        .
    \end{align*}
    Now, using $\Delta v=\frac{d}{2}P$ and rearranging the above equation, we derive
    \begin{align*}
        \frac{1}{2}v\Delta P+\frac{2-d}{2}\langle\nabla P,\nabla v\rangle
        =\left|\nabla^2 v-\frac{\Delta v}{n}g\right|^2
        +\left(\frac{1}{n}-\frac{1}{d}\right)(\Delta v)^2
        +Ric(\nabla v,\nabla v)
        .
    \end{align*}
    Multiply both sides by $v^{1-d}$ to get the following divergence-type equation:
    \begin{align}\label{eq. div 1_1-d}
        \frac{1}{2}div(v^{2-d}\nabla P)
        =v^{1-d}\left(
        \left|\nabla^2 v-\frac{\Delta v}{n}g\right|^2
        +\left(\frac{1}{n}-\frac{1}{d}\right)(\Delta v)^2
        +Ric(\nabla v,\nabla v)
        \right).
    \end{align}
    Now use \cref{eq. nabla Delta v} to substitute $\nabla P$ in \cref{eq. div 1_1-d} and get
    \begin{align}\label{eq. warped div 1_1-d_new}
        \mathrm{div}\left(
        v^{1-d}\left(\nabla_{\nabla v}\nabla v-\frac{\Delta v}{d}\nabla v\right)
        \right)
        =v^{1-d}\left(
        \left|\nabla^2 v-\frac{\Delta v}{n}g\right|^2
        +\left(\frac{1}{n}-\frac{1}{d}\right)(\Delta v)^2
        +Ric(\nabla v,\nabla v)
        \right).
    \end{align}
    Integrating this equation over $M$ and using the boundary value condition \cref{eq. PDE of v on warped}, we obtain
    \begin{align}\label{eq. warped 1}
        &\int_{M}v^{1-d}\left(
        \left|\nabla^2 v-\frac{\Delta v}{n}g\right|^2
        +\left(\frac{1}{n}-\frac{1}{d}\right)(\Delta v)^2
        +Ric(\nabla v,\nabla v)
        \right)\notag\\
        =&\int_{\partial M}f^{1-d}
        \left(\nabla^2 v-\frac{\Delta v}{d}g\right)(\nabla v,\nu)\notag
        =\int_{\partial M}f^{1-d}
        \left(\nabla^2 v-\frac{\Delta v}{d}g\right)(\nabla f+\chi\nu,\nu)\\
        =&\int_{\partial M}f^{1-d}
        \left(\langle \nabla f,\nabla \chi\rangle
        -\frac{\rho'(R)}{\rho(R)}|\nabla f|^2\right)
        +\int_{\partial M}\chi f^{1-d}
        \left(\nabla^2 v-\frac{\Delta v}{d}g\right)(\nu,\nu)\notag\\
        =&\left(\frac{2}{d-2}\lambda-\frac{\rho'(R)}{\rho(R)}\right)
        \int_{\mathbb{S}^{n-1}}f^{1-d}|\nabla f|^2
        +\frac{2}{d-2}\lambda
        \int_{\partial M}f^{2-d}
        \left(\nabla^2 v-\frac{\Delta v}{d}g\right)(\nu,\nu)\notag\\
        &-\frac{2}{d-2}
        \int_{\partial M}f^{1-d}
        \left(\nabla^2 v-\frac{\Delta v}{d}g\right)(\nu,\nu),
    \end{align}
    where in the penultimate equality, we used $\langle \nabla f,\nabla_{\nabla f}\nu\rangle
    =\frac{\rho'(R)}{\rho(R)}|\nabla f|^2 $.

    \textbf{Step 2:} Consider the weight function 
        \begin{align}\label{eq. weight in warped case}
            w(x):=\int_{r(x)}^R\rho(t)\dd t.
        \end{align}
        Then $w$ satisfies: $w|_{\partial M}=0,\ \nabla w|_{\partial M}=-\rho(R)\nu,\ \Delta w=-n\rho'$. A crucial fact is that $\nabla w$ is a closed conformal vector field, i.e.
        \begin{align*}
            \nabla^2 w=\frac{\Delta w}{n}g\quad \text{ on }M.
        \end{align*}

    We shall utilize $w$ to tackle the last two terms in the right hand side of \cref{eq. warped 1}.

    \cref{cor. div 1w} and \cref{eq. nabla Delta v} yield that
        \begin{align}\label{eq. warped div hessian v}
            div\left(\nabla_{\nabla w}\nabla v-\frac{\Delta v}{d}\nabla w\right)
            =&\left(\frac{n}{d}-1\right)\rho'\Delta v
            +(n-1)\rho''\frac{\partial v}{\partial r}
            +\left(1-\frac{1}{d}\right)\langle\nabla \Delta v,\nabla w\rangle\notag\\
            =&\left(\frac{n}{d}-1\right)\rho'\Delta v
            +(n-1)\rho''\frac{\partial v}{\partial r}
            +(d-1)v^{-1}\left\langle
             \nabla_{\nabla w}\nabla v-\frac{\Delta v}{d}\nabla w,\nabla v\right\rangle.
        \end{align}
        Therefore,
        \begin{align}\label{eq. warped div 1w_2-d}
        &\mathrm{div}\left(
        v^{2-d}
        \left(\nabla_{\nabla w}\nabla v-\frac{\Delta v}{d}\nabla w\right)
        \right)\notag\\
        =&\left(\frac{n}{d}-1\right)
        \rho'v^{2-d}\Delta v
        +(n-1)\rho''v^{2-d}\frac{\partial v}{\partial r}
        +v^{1-d}
        \left\langle\nabla_{\nabla v}\nabla v-\frac{\Delta v}{d}\nabla v,\nabla w\right\rangle.
    \end{align}

    It follows from \cref{eq. warped div 1_1-d_new} and \cref{eq. warped div 1w_2-d} that
    \begin{align*}
        &wv^{1-d}\left(
        \left|\nabla^2 v-\frac{\Delta v}{n}g\right|^2
        +\left(\frac{1}{n}-\frac{1}{d}\right)(\Delta v)^2
        +Ric(\nabla v,\nabla v)
        \right)
        =w\mathrm{div}\left(
        v^{1-d}\left(\nabla_{\nabla v}\nabla v-\frac{\Delta v}{d}\nabla v\right)
        \right)\notag\\
        =&\mathrm{div}\left(
        wv^{1-d}
        \left(\nabla_{\nabla v}\nabla v-\frac{\Delta v}{d}\nabla v\right)
        \right)
        -v^{1-d}\left\langle\nabla_{\nabla v}\nabla v-\frac{\Delta v}{d}\nabla v,\nabla w\right\rangle\notag\\
        =&\mathrm{div}\left(
        wv^{1-d}\left(\nabla_{\nabla v}\nabla v-\frac{\Delta v}{d}\nabla v\right)
        \right)
        -\mathrm{div}\left(
        v^{2-d}
        \left(\nabla_{\nabla w}\nabla v-\frac{\Delta v}{d}\nabla w\right)
        \right)\\
        &+\left(\frac{n}{d}-1\right)\rho'v^{2-d}\Delta v
        +(n-1)\rho''v^{2-d}\frac{\partial v}{\partial r}.
    \end{align*}
    Integrating this equation over $M$ and noticing that $\nabla w|_{\partial M}=-\rho(R)\nu$, we get

    \begin{align}\label{eq. warped 1w_2-d}
        &\int_{M}wv^{1-d}\left(
        \left|\nabla^2 v-\frac{\Delta v}{n}g\right|^2
        +\left(\frac{1}{n}-\frac{1}{d}\right)(\Delta v)^2
        +Ric(\nabla v,\nabla v)
        \right)\notag\\
        =&\rho(R)\int_{\partial M}f^{2-d}\left(\nabla^2 v-\frac{\Delta v}{d}g\right)(\nu,\nu)
        +\left(\frac{n}{d}-1\right)
        \int_{M}\rho'v^{2-d}\Delta v
        +(n-1)\int_M\rho''v^{2-d}\frac{\partial v}{\partial r}.
    \end{align}

    Once again, \cref{eq. nabla Delta v} and \cref{cor. div 1w} imply that
    \begin{align*}
        \mathrm{div}\left(
        v^{1-d}
        \left(\nabla_{\nabla w}\nabla v-\frac{\Delta v}{d}\nabla w\right)
        \right)
        =\left(\frac{n}{d}-1\right)
        \rho'v^{1-d}\Delta v
        +(n-1)\rho''v^{1-d}\frac{\partial v}{\partial r}.
    \end{align*}
    Integrating this equation over $M$, we get
    \begin{align}\label{eq. warped 1w_1-d}
        \rho(R)\int_{\partial M}f^{1-d}
        \left(\nabla^2 v-\frac{\Delta v}{d}g\right)(\nu,\nu)
        =\left(1-\frac{n}{d}\right)
        \int_{M}\rho'v^{1-d}\Delta v
        -(n-1)\int_M \rho''v^{1-d}\frac{\partial v}{\partial r}.
    \end{align}

    Note that for $(M^n,g)=\mathbb{B}^n$ and $q=\frac{n}{n-2}$, \cref{eq. warped 1w_1-d} is exactly the Pohozaev identity \cite[Proposition 1.4]{Sch88} used in Escobar's proof of the critical power case \cite[Theorem 2.1]{Esc90}, whose validity comes from the fact that $\nabla w$ is a closed conformal vector field.

    Plugging \cref{eq. warped 1w_2-d} and \cref{eq. warped 1w_1-d} into \cref{eq. warped 1}, we derive that
    \begin{align}\label{eq. subkey integral identity}
        &\int_M \left(1-\frac{2}{d-2}\lambda\frac{1}{\rho(R)}w\right)
        v^{1-d}\left(
        \left|\nabla^2v-\frac{\Delta v}{n}g\right|^2
        +\left(\frac{1}{n}-\frac{1}{d}\right)(\Delta v)^2
        +Ric(\nabla v,\nabla v)
        \right)\notag\\
        =&\left(\frac{2}{d-2}\lambda-\frac{\rho'(R)}{\rho(R)}\right)\int_{\partial M}f^{1-d}|\nabla f|^2\notag\\
        &
        +\frac{2}{d-2}\lambda\frac{1}{\rho(R)}
        \left(1-\frac{n}{d}\right)\int_M \rho' v^{2-d}\Delta v
        -\frac{2}{d-2}\lambda\frac{1}{\rho(R)}(n-1)\int_M \rho''v^{2-d}\frac{\partial v}{\partial r}\notag\\
        &-\frac{2}{d-2}\frac{1}{\rho(R)}
        \left(1-\frac{n}{d}\right)\int_M \rho' v^{1-d}\Delta v
        +\frac{2}{d-2}\frac{1}{\rho(R)}(n-1)\int_M \rho'' v^{1-d}\frac{\partial v}{\partial r}.
    \end{align}

    Further notice that by divergence theorem and \cref{eq. PDE of v on warped}, we have
    \begin{align*}
            &\int_M \rho'' v^{2-d}\frac{\partial v}{\partial r}
            =\int_M div\left(\rho'v^{2-d}\nabla v\right)
            -\rho' v^{2-d}\Delta v-(2-d)\rho'v^{1-d}|\nabla v|^2\\
            =&\rho'(R)\int_{\partial M}f^{2-d}\chi
            +\frac{d-4}{d}\int_M \rho'v^{2-d}\Delta v\\
            =&\frac{d-4}{d}\int_M \left(\rho'-\rho'(R)\right)v^{2-d}\Delta v,
    \end{align*}
    where in the last line, we used $\int_{\partial M}f^{2-d}\chi
        =\int_{M}\mathrm{div}(v^{2-d}\nabla v)
        =\frac{4-d}{d}\int_{M}v^{2-d}\Delta v$.
    Similarly we have
    \begin{align}\label{eq. integral trick}
        \int_M \rho'' v^{1-d}\frac{\partial v}{\partial r}
            =\frac{d-2}{d}\int_M (\rho'-\rho'(R))v^{1-d}\Delta v.
    \end{align}
    Inserting these two formulas into \cref{eq. subkey integral identity}, we finally arrive at our key integral identity
    \begin{align}\label{eq. key integral identity}
        &\int_M \left(1-\frac{2}{d-2}\lambda\frac{1}{\rho(R)}w\right)
        v^{1-d}\left(
        \left|\nabla^2v-\frac{\Delta v}{n}g\right|^2
        +\left(\frac{1}{n}-\frac{1}{d}\right)(\Delta v)^2
        +Ric(\nabla v,\nabla v)
        \right)\notag\\
        =&\left(\frac{2}{d-2}\lambda-\frac{\rho'(R)}{\rho(R)}\right)\int_{\partial M}f^{1-d}|\nabla f|^2\notag\\
        &
        +\frac{2}{d-2}\lambda\frac{1}{\rho(R)}
        \int_M
        \left(
        \frac{(d-4)(n-1)}{d}\rho'(R)-\frac{dn-2d-3n+4}{d}\rho'
        \right)
         v^{2-d}\Delta v\notag
        \\
        &+\frac{2}{\rho(R)}\int_M 
        \left(\frac{(d-1)(n-2)}{d(d-2)}\rho'
        -\frac{n-1}{d}\rho'(R)
        \right)
        v^{1-d}\Delta v.
    \end{align}

    Notice that by \cref{ineq. R leq 1} and \cref{ineq. rho r leq rho R}, we have
    \begin{align*}
        w(r(x))\leq (R-r)\rho(R)\leq \rho(R),\quad
        \forall x\in M.
    \end{align*}
    It follows from this and our hypothesis on $\lambda$ that 
    \begin{align*}
        1-\frac{2}{d-2}\lambda \frac{w}{\rho(R)}
    \geq1-\frac{2}{d-2}\lambda
    \geq 1-C(n,q,\rho'(R))\geq 0,
    \end{align*}
    where $C(n,q,\rho'(R))\in(0,1]$ is defined in \cref{def. C}.
    Hence the left hand side of \cref{eq. key integral identity} is non-negative. We aim to prove that the right hand side of \cref{eq. key integral identity} is non-positive.
    
    For this purpose, note that $\frac{2}{d-2}\lambda-\frac{\rho'(R)}{\rho(R)}
    \leq C(n,q,\rho'(R))-\frac{\rho'(R)}{\rho(R)}
    \leq 0$, it suffices to show that
    \begin{align*}
        &\lambda
        \int_M
        \left(
        (d-4)(n-1)\rho'(R)-(dn-2d-3n+4)\rho'
        \right)
         v^{2-d}\Delta v\notag
        \\
        +&\int_M 
        \left((d-1)(n-2)\rho'
        -(d-2)(n-1)\rho'(R)
        \right)
        v^{1-d}\Delta v
        \leq 0.
    \end{align*}
    Furthermore,  using $\rho'\leq 1$ (see  \cref{ineq. rho prime r}), it suffices to show that
    \begin{align}\label{eq. warped goal}
        &\lambda
        \int_M
        \left(
        (d-4)(n-1)\rho'(R)-(dn-2d-3n+4)\rho'
        \right)
         v^{2-d}\Delta v\notag
        \\
        +&\left((d-1)(n-2)
        -(d-2)(n-1)\rho'(R)
        \right)
        \int_M 
        v^{1-d}\Delta v
        \leq 0.
    \end{align}
    The proof of \cref{eq. warped goal} will be established through steps 3 to 5.

    \textbf{Step 3:} We shall use the boundary condition \cref{eq. PDE of v on warped} to make \cref{eq. warped goal} homogeneous.
    
    The boundary  condition in \cref{eq. PDE of v on warped} implies that
    \begin{align}
        \int_{\partial M}f^{2-d}\chi
        =\int_{M}\mathrm{div}(v^{2-d}\nabla v)
        =\int_{M}(v^{2-d}\Delta v+(2-d)v^{1-d}|\nabla v|^2)
        =\frac{4-d}{d}\int_{M}v^{2-d}\Delta v,\notag\\
        \int_{\partial M}f^{1-d}\chi
        =\int_{M}\mathrm{div}(v^{1-d}\nabla v)
        =\int_{M}(v^{1-d}\Delta v+(1-d)v^{-d}|\nabla v|^2)
        =\frac{2-d}{d}\int_{M}v^{1-d}\Delta v.\label{eq. f to 1-d_ chi}
    \end{align}
    Therefore we have
    \begin{align*}
        \int_{\partial M}f^{1-d}\chi^2
        =\frac{2}{d-2}\lambda\int_{\partial M}f^{2-d}\chi
        -\frac{2}{d-2}\int_{\partial M}f^{1-d}\chi
        =\frac{2(4-d)}{d(d-2)}\lambda\int_{M} v^{2-d}\Delta v+\frac{2}{d}\int_{M}v^{1-d}\Delta v.
    \end{align*}
    Equivalently,
    \begin{align}\label{eq. homogenization}
        \int_{M}v^{1-d}\Delta v
        =\frac{d}{2}\int_{\partial M}f^{1-d}\chi^2
        +\frac{d-4}{d-2}\lambda\int_{M}v^{2-d}\Delta v.
    \end{align}
    It follows from \cref{eq. homogenization} that \cref{eq. warped goal}  is equivalent to
    \begin{align}\label{ineq. warped goal 2}
        &\left(
        (d-1)(n-2)-(d-2)(n-1)\rho'(R)
        \right)
        \frac{d}{2}\int_{\partial M} f^{1-d}\chi^2\notag\\
        +&\lambda
        \int_M
        \left(
        \frac{(d-1)(d-4)(n-2)}{d-2}
        -((d-n)(n-2)+(n-4)(n-1))\rho'
        \right)
        v^{2-d}\Delta v\leq 0.
    \end{align}

    \textbf{Step 4:} We shall derive a lower bound estimate of $\int_{\partial M}f^{1-d}\chi^2$. Using $\nabla^2 w=-\rho'g,\ \nabla w|_{\partial M}=-\rho(R)\nu$ and \cref{eq. PDE of v on warped}, we get
    \begin{align*}
        &\rho(R)\int_{\partial M}f^{1-d}\chi^2
        =-\int_{M}\mathrm{div}\left(v^{1-d}\langle \nabla v,\nabla w\rangle\nabla v\right)\notag\\
        =&-\int_{M}\left\{v^{1-d}\langle \nabla v,\nabla w\rangle\Delta v+(1-d)v^{-d}\langle \nabla v,\nabla w\rangle|\nabla v|^2
        +v^{1-d}\langle\nabla_{\nabla v}\nabla v,\nabla w\rangle
        +v^{1-d}\nabla^2 w(\nabla v,\nabla v)\right\}\notag\\
        =&-\int_{M}v^{1-d}
        \left\langle\nabla_{\nabla v}\nabla v-\frac{d-2}{d}\Delta v\nabla v,\nabla w\right\rangle
        +\frac{2}{d}\int_{M}\rho'v^{2-d}\Delta v\notag\\
        =&\int_{M}
        w\mathrm{div}\left(
        v^{1-d}
        \left(\nabla_{\nabla v}\nabla v-\frac{d-2}{d}\Delta v\nabla v\right)
        \right)
        +\frac{2}{d}\int_{M}\rho'v^{2-d}\Delta v,
    \end{align*}
    where in the last line, we used the divergence theorem and $w|_{\partial M}=0.$
    Furthermore, we use \cref{cor.div 1} and \cref{eq. nabla Delta v} to calculate the first integrand in the above integrals:
    \begin{align*}
        &w\mathrm{div}\left(
        v^{1-d}
        \left(\nabla_{\nabla v}\nabla v-\frac{d-2}{d}\Delta v\nabla v\right)
        \right)\\
        =&wv^{1-d}\left(
            \mathrm{div}
            \left(\nabla_{\nabla v}\nabla v-\frac{d-2}{d}\Delta v\nabla v\right)
            +(1-d)v^{-1}
            \left\langle \nabla_{\nabla v}\nabla v-\frac{d-2}{d}\Delta v\nabla v,\nabla v\right\rangle
        \right)\\
        =&wv^{1-d}\left(
        \left|\nabla^2 v-\frac{\Delta v}{n}g\right|^2
        +\left(\frac{1}{n}-\frac{d-2}{d}\right)
        (\Delta v)^2
        +2v^{-1}
        \left\langle \nabla_{\nabla v}\nabla v-\frac{\Delta v}{d}\nabla v,\nabla v\right\rangle\right.\\
        &\quad\quad\quad\left.
        +(1-d)v^{-1}\left\langle \nabla_{\nabla v}\nabla v-\frac{d-2}{d}\Delta v\nabla v,\nabla v\right\rangle
        \right)\\
        =&wv^{1-d}\left(
        \left|\nabla^2 v-\frac{\Delta v}{n}g\right|^2
        +(3-d)v^{-1}\langle \nabla_{\nabla v}\nabla v,\nabla v\rangle
        +\left(\frac{1}{n}-\frac{d-2}{d}+\frac{2(d-3)}{d}\right)(\Delta v)^2
        \right)\\
        =&wv^{1-d}\left(
        \left|\nabla^2 v-\frac{\Delta v}{n}g
        +\frac{3-d}{2}
        \left(
        \frac{\mathrm{d}v\otimes\mathrm{d}v}{v}
        -\frac{1}{n}\frac{|\nabla v|^2}{v}g
        \right)\right|^2
        +\frac{2n(d-n)+(n-3)(2n-3)}{nd^2}(\Delta v)^2
        \right)\geq 0,
    \end{align*}
    where in the last line, we used $|\frac{\mathrm{d}v\otimes\mathrm{d}v}{v}
        -\frac{1}{n}\frac{|\nabla v|^2}{v}g|^2=\frac{n-1}{n}v^{-2}|\nabla v|^4$, the equation \cref{eq. PDE of v on warped} and
    \begin{align*}
        v^{-1}\langle\nabla_{\nabla v}\nabla v,\nabla v\rangle
        =\left\langle \nabla^2 v-\frac{\Delta v}{n}g,
        \frac{\dd v\otimes \dd v}{v}-\frac{1}{n}\frac{|\nabla v|^2}{v}g\right\rangle
        +\frac{2}{dn}(\Delta v)^2.
    \end{align*}

    In conclusion, we showed that
    \begin{align}\label{ineq. integral of f chi squared}
        \rho(R)\int_{\partial M}f^{1-d}\chi^2
        \geq 
        \frac{2}{d}\int_{M}\rho'v^{2-d}\Delta v.
    \end{align}

    \textbf{Step 5:} We shall show that \cref{ineq. warped goal 2} holds true. By our assumption $1<q\leq \frac{2(n-1)\rho'(R)-(n-2)}{n-2}$ and $d=\frac{2q}{q-1}$, it is straightforward to see that 
    \begin{align}\label{ineq. rho' R lower bound}
        \rho'(R)\geq\frac{(d-1)(n-2)}{(d-2)(n-1)}.
    \end{align}
    Therefore, we could use \cref{ineq. integral of f chi squared} to get an upper bound of \cref{ineq. warped goal 2} as follows:
    \begin{align}\label{ineq. upper bound of 3_23}
        &\left(
        (d-1)(n-2)-(d-2)(n-1)\rho'(R)
        \right)
        \frac{d}{2}\int_{\partial M} f^{1-d}\chi^2\notag\\
        &+\lambda
        \int_M
        \left(
        \frac{(d-1)(d-4)(n-2)}{d-2}
        -((d-n)(n-2)+(n-4)(n-1))\rho'
        \right)
        v^{2-d}\Delta v \notag\\
        \leq
        &\int_M \left(
        \frac{(d-1)(d-4)(n-2)}{d-2}\lambda
        +A(d,n,R,\lambda)\rho'\right)v^{2-d}\Delta v,
    \end{align}
    where
    \begin{align*}
        A(d,n,R,\lambda):=\frac{(d-1)(n-2)}{\rho(R)}
        -(d-2)(n-1)\frac{\rho'(R)}{\rho(R)}
        -\left(
        (d-n)(n-2)+(n-4)(n-1)
        \right)\lambda.
    \end{align*}
    For $1<q\leq \frac{2(n-1)\rho'(R)-(n-2)}{n-2}$, define the constant alluded in our \cref{thm. C. warped product} as:
    \begin{align}\label{def. C}
         C(n,q,\rho'(R))
            :=&\begin{cases}
                \left(1-\frac{2(n-1)}{n-(n-2)q}(1-\rho'(R))\right),
                &\text{if } n=3 \text{ and }  
                 q>2,\\
                \left(1-\frac{2(n-1)}{n-(n-2)q}(1-\rho'(R))\right)
                \frac{(q-1)(n-(n-2)q)\rho'(R)}{C_1(3,q,\rho'(R))},
                &\text{if } n=3 \text{ and }  
                 1<q\leq 2,\\
                \left(1-\frac{2(n-1)}{n-(n-2)q}(1-\rho'(R))\right)
                \frac{(q-1)(n-(n-2)q)\rho'(R)}{C_1(n,q,\rho'(R))},
                &\text{if } n\geq 4,
            \end{cases}
    \end{align}
    where
    \begin{align*}
        C_1(n,q,\rho'(R)):=(q-1)(n-(n-2)q)\rho'(R)+(2-q)(q+1)(n-2)(1-\rho'(R))
    \end{align*}
    

    We claim that for $0<\lambda\leq \frac{1}{q-1}\frac{1}{\rho(R)}C(n,q,\rho'(R))$ (note that this range covers our hypothesis on $\lambda$ since $\rho(R)\leq 1$), there holds
    \begin{align*}
        \frac{(d-1)(d-4)(n-2)}{d-2}\lambda
        +A(d,n,R,\lambda)\rho'(r(x))\leq 0,\quad \forall x\in M,
    \end{align*}
    from which \cref{ineq. warped goal 2} follows since we have \cref{ineq. upper bound of 3_23}.

    In fact, sine $\rho'(R)\leq \rho'(r(x))\leq 1$ (see \cref{ineq. rho prime r}), it suffices to show that the following two algebraic inequalities hold for $0<\lambda\leq \frac{1}{q-1}\frac{1}{\rho(R)}C(n,q,\rho'(R))$:
    \begin{align}
        \frac{(d-1)(d-4)(n-2)}{d-2}\lambda
        +A(d,n,R,\lambda)\rho'(R)\leq 0,\label{ineq. algebraic 1}\\
        \frac{(d-1)(d-4)(n-2)}{d-2}\lambda
        +A(d,n,R,\lambda)\leq 0.\label{ineq. algebraic 2}
    \end{align}
    Now we verify these two inequalities. Denote $\epsilon:=1-\rho'(R)$. By \cref{ineq. rho' R lower bound}, we get
    \begin{align*}
        0\leq\epsilon\leq \frac{d-n}{(d-2)(n-1)}.
    \end{align*}
    $\bullet$ For \cref{ineq. algebraic 1},
        \begin{align*}
            &\frac{(d-1)(d-4)(n-2)}{d-2}\lambda
        +A(d,n,R,\lambda)\rho'(R)\\
        =&\frac{(d-1)(d-4)(n-2)}{d-2}(\rho'(R)+\epsilon)\lambda
       -((d-n)(n-2)+(n-4)(n-1))\rho'(R)\lambda\\
        &+
        ((d-1)(n-2)-(d-2)(n-1)(1-\epsilon))
        \frac{\rho'(R)}{\rho(R)}\\
        =&\frac{1}{d-2}\Big(2(d-n)\rho'(R)+(d-1)(d-4)(n-2)\epsilon \Big)\lambda
        -\Big(d-n-(d-2)(n-1)\epsilon\Big)
        \frac{\rho'(R)}{\rho(R)}\\
        \leq& 0,
        \end{align*}
        where the last inequality holds for $\lambda$ satisfying
        \begin{align*}
        \begin{cases}
            \lambda
            \leq &\!\!\!\!\!\!\!\!\!\!
            \frac{d-2}{2}\frac{1}{\rho(R)}
            \left(1-\frac{(d-2)(n-1)}{d-n}\epsilon\right)
            \frac{2(d-n)\rho'(R)}{2(d-n)\rho'(R)+(d-1)(d-4)(n-2)\epsilon}\\
            &\!\!\!\!\!\!\!\!\!\!\!
            = \frac{1}{q-1}\frac{1}{\rho(R)}
            \left(1-\frac{2(n-1)}{n-(n-2)q}(1-\rho'(R))\right)\frac{(q-1)(n-(n-2)q)\rho'(R)}{C_1(n,q,\rho'(R))},
            \text{ if } C_1(n,q,\rho'(R))> 0,\\
            \lambda>0,& \quad\quad\quad\quad\quad\quad\quad\quad\quad\quad\quad\quad\quad\quad\quad\quad\quad\quad\quad\quad\quad
            \ \ \text{ if } C_1(n,q,\rho'(R))\leq 0.
        \end{cases}
        \end{align*}        

        $\bullet$ For \cref{ineq. algebraic 2},
        \begin{align*}
         &\frac{(d-1)(d-4)(n-2)}{d-2}\lambda
        +A(d,n,R,\lambda)\\
        =&\frac{(d-1)(d-4)(n-2)}{d-2}\lambda
        -((d-n)(n-2)+(n-4)(n-1))\lambda\\
        &+\Big((d-1)(n-2)-(d-2)(n-1)(1-\epsilon)\Big)\frac{1}{\rho(R)}\\
        =&\frac{2(d-n)}{d-2}\lambda
        -\Big(d-n-(d-2)(n-1)\epsilon\Big)\frac{1}{\rho(R)}\\
        \leq &0,
        \end{align*}
        where the last inequality is equivalent to
        \begin{align*}
            \lambda\leq &
            \frac{d-2}{2}\frac{1}{\rho(R)}
            \left(
            1-\frac{(d-2)(n-1)}{d-n}\epsilon
            \right)
            =\frac{1}{q-1}\frac{1}{\rho(R)}
            \left(1-\frac{2(n-1)}{n-(n-2)q}(1-\rho'(R))\right).
        \end{align*}
        
        In conclusion, for $1<q\leq \frac{2(n-1)\rho'(R)-(n-2)}{n-2}$ and $\lambda $ satisfying
        \begin{align*}
            0<\lambda\leq
            \begin{cases}
                \frac{1}{q-1}\frac{1}{\rho(R)}
            \left(1-\frac{2(n-1)}{n-(n-2)q}(1-\rho'(R))\right)\times
            \min\{
            \frac{(q-1)(n-(n-2)q)\rho'(R)}{C_1(n,q,\rho'(R))},
            1
            \},
            &\text{if } C_1(n,q,\rho'(R))>0,\\
            \frac{1}{q-1}\frac{1}{\rho(R)}
            \left(1-\frac{2(n-1)}{n-(n-2)q}(1-\rho'(R))\right),
            &\text{if } C_1(n,q,\rho'(R))\leq 0,
            \end{cases}
        \end{align*}
        \cref{ineq. algebraic 1} and \cref{ineq. algebraic 2} are fulfilled.
        
        Furthermore, note that if $C_1(n,q,\rho'(R))\geq 0$, then
        \begin{align*}
            \min\left\{
            \frac{(q-1)(n-(n-2)q)\rho'(R)}{C_1(n,q,\rho'(R))},
            1
            \right\}
            =\frac{(q-1)(n-(n-2)q)\rho'(R)}{C_1(n,q,\rho'(R))}
            \Longleftrightarrow
            q\leq 2,
        \end{align*}
        which always holds when $n\geq 4$, since $1<q\leq \frac{2(n-1)\rho'(R)-(n-2)}{n-2}\leq \frac{n}{n-2}$. Therefore we could simplify the requirement on $\lambda$ as 
        \begin{align*}
            0<\lambda\leq \frac{1}{q-1}\frac{1}{\rho(R)}C(n,q,\rho'(R)),
        \end{align*}
        where $C(n,q,\rho'(R))$ is defined in \cref{def. C}.

        \textbf{Step 6:} In summary, we have proved through steps 3 to 5 that the left hand side of our key integral identity \cref{eq. key integral identity} is nonnegative, and the right hand side is nonpositive (see \cref{eq. warped goal} and \cref{ineq. warped goal 2}). This forces
        \begin{align}\label{eq. rigidity}
            0=&\int_M \left(1-\frac{2}{d-2}\lambda\frac{1}{\rho(R)}w\right)
        v^{1-d}\left(
        \left|\nabla^2v-\frac{\Delta v}{n}g\right|^2
        +\left(\frac{1}{n}-\frac{1}{d}\right)(\Delta v)^2
        +Ric(\nabla v,\nabla v)
        \right)\notag\\
        =&\left(\frac{2}{d-2}\lambda-\frac{\rho'(R)}{\rho(R)}\right)\int_{\partial M}f^{1-d}|\nabla f|^2\notag\\
        &
        +\frac{2}{d-2}\lambda\frac{1}{\rho(R)}
        \int_M
        \left(
        \frac{(d-4)(n-1)}{d}\rho'(R)-\frac{dn-2d-3n+4}{d}\rho'
        \right)
         v^{2-d}\Delta v\notag
        \\
        &+\frac{2}{\rho(R)}\int_M 
        \left(\frac{(d-1)(n-2)}{d(d-2)}\rho'
        -\frac{n-1}{d}\rho'(R)
        \right)
        v^{1-d}\Delta v.
        \end{align}

        If $\rho'(R)<1$, then $q\leq \frac{2(n-1)\rho'(R)-(n-2)}{n-2}<\frac{n}{n-2}$. It follows that $d:=\frac{2q}{q-1}>n$. 
        Then by \cref{eq. rigidity}, we get
         $(\frac{d}{2}v^{-1}|\nabla v|^2)^2=(\Delta v)^2\equiv 0$. Therefore $v$ is a constant function and so is $u$.

        If $\rho'(R)=1$, then $\rho'\equiv 1$ and 
        \begin{align*}
            \rho(r)=r,\quad \forall r\in[0,R].
        \end{align*}
        In other words, $(M^n,g)$ is isometric to a round Euclidean ball of radius $R$. If $R<1$, then by our hypothesis  $\lambda\leq \frac{1}{q-1}C(n,q,\rho'(R))\leq \frac{1}{q-1}=\frac{d-2}{2}$, \cref{eq. rigidity} and \cref{eq. warped goal} we have
        \begin{align*}
            \int_{\partial M}f^{1-d}|\nabla f|^2=0,
        \end{align*}
        so $v|_{\partial M}=f$ is a constant function. Therefore, $u$ is a harmonic function on $(M^n,g)$ with a constant boundary value. It follows that $u$ is a constant.

        Finally, the remaining case is that $(M^n,g)$ is isometric to $\mathbb{B}^n$, on which the solution $u$ is completely classified by deducing that the P-function $P=\frac{2}{d}\Delta v$ is a constant and some further calculation. You could refer to \cite[Page 3515-3516]{GL25} for the details.
\end{proof}
\begin{remark}
    Note that $0< C(n,q,\rho'(R))\leq 1$ and $C(n,q,1)\equiv 1, \forall n\geq 3, 1<q\leq \frac{n}{n-2}$, so we would recover \cref{thm. GL25} by taking $\rho(r)=r,\forall r\in[0,1]$. Moreover, in this model case, the computations are much simpler than those presented above and more streamlined compared to the original approach in \cite{GL25}.
\end{remark}
\begin{remark}
    We now provide a review of the parameter choices in our proof, accompanied by a comparison with those adopted by Gu-Li \cite{GL25}.
    They choose the power of the test function as $a=-\frac{q+1}{q-1}$ (see \cite[Section 3.1]{GL25}), and this is exactly $1-d$ in our \cref{eq. warped div 1_1-d_new}, which is chosen to get a divergence-type equation of the P-function \cref{eq. div 1_1-d}. 
    They choose the combination coefficient of their vector fields as $b=\frac{q-1}{2q}$ (see \cite[Section 3.2]{GL25}), and this is exactly $\frac{1}{d}$ in our \cref{eq. nabla Delta v}, which naturally appears when calculating the gradient of the P-function. 
    They finally choose the combination coefficient of the weight function as $\frac{2}{1+c}=\lambda(q-1)$ (see \cite[Section 3.3]{GL25}), which equals $\frac{2}{d-2}\lambda$ and appears naturally by our \cref{eq. warped 1} and \cref{eq. warped 1w_2-d}.
\end{remark}

It is worth comparing our \cref{thm. C. warped product} with Guo-Hang-Wang's \cref{thm. GHW21}.

First, note that on the warped product manifold $(M^n,g)=([0,R]\times \mathbb{S}^{n-1},dr^2+\rho(r)^2dS_{n-1}^2)$ with $Ric_g\geq 0$, the radial-direction sectional curvature is nonnegative:
\begin{align}\label{ineq. sec curvature}
    \mathrm{Sec}(\nabla r,X)=-\frac{\rho''}{\rho}\geq 0,
    \quad \forall X\in \mathfrak{X}(M)
    \text{ with } |X|=1, X\perp \nabla r.
\end{align}
In this circumstance, the weight function used in \cite{GHW21} is
\begin{align}\label{eq. GHW weight function}
    \tilde{w}(x)
    =d(x,\partial M)-\frac{d^2(x,\partial M)}{2}
    =-\frac{1}{2}r^2(x)+(R-1)r(x)+R-\frac{1}{2}R^2.
\end{align}
Combining \cref{ineq. sec curvature} with the second fundamental form hypothesis $\Pi\geq 1$, we could get
\begin{align*}
    \nabla^2 \tilde{w}=-\dd r\otimes \dd r+(R-r-1)\frac{\rho'(r)}{\rho(r)}(g-\dd r\otimes \dd r)
    \leq -g,
\end{align*}
which follows from the comparison theorem \cite[Theorem 2.31]{Kas82} (or a direct ODE comparison in this special case). Therefore, for the warped product case,  the arguments of \cite{GHW21} still work with the weight function \cref{eq. GHW weight function} and would derive a classification result for positive solutions to \cref{eq. equation of u on M} with a \emph{sharp} range for $\lambda$ when $3\leq n\leq 8$ and $1<q\leq \frac{4n}{5n-9}$. We remark that the sharpness of $\lambda$ is illustrated by an existence result on the model space $\mathbb{B}^n$ \cite[Proposition 1]{GHW21} and the sharp value of $\lambda$ determines the sharp constant of the Sobolev trace inequality (see \cite[Corollary 1]{GHW21}).

In contrast, our \cref{thm. C. warped product} emphasizes the role of a closed conformal vector field and uses a different weight function \cref{eq. weight in warped case}. While this approach yields a range for $\lambda$ that is not necessarily sharp, it has the advantage of being applicable in any dimension $n\geq 3$. Moreover, the range for $q$ is sharp when $\rho'(R)=1$.
 This suggests the possibility of achieving a \emph{sharp} range for $\lambda$ in \emph{all} dimensions by merging the two approaches, possibly subject to conditions on
 $q$ and $\rho'(R)$. 
 The case in \cref{thm.C. minimizer} may offer some initial support for this  belief.



\section{Proof of \cref{thm.C. minimizer}}\label{Sec.4}
In this section, we address the classification of minimizers for the Sobolev trace inequality on warped product manifolds and prove \cref{thm.C. minimizer}. Compared to \cref{thm. C. warped product}, the present result combines the minimizing property with a sharp Sobolev inequality on the boundary $\partial M$ to control certain boundary integrals (see Step 3).
\begin{proof}[Proof of \cref{thm.C. minimizer}:]
    \textbf{Step 1:}
        If $\rho'(R)<1$, then $1<q\leq \frac{2(n-1)\rho'(R)-(n-2)}{n-2}<\frac{n}{n-2}$ and the minimizer $u_0$ exists by a classical compactness argument. If $\rho'(R)=1$ and for $q=\frac{n}{n-2}$, then $(M^n,g)=\mathbb{B}^n(R)$ and the minimizer $u_0$ still exists by the concentration compactness principle \cite{Lio85} (see \cite[Page 693]{Esc88}).
        
            W.L.O.G., assume $\int_{\partial M}u_0^{q+1}=1$. Then the minimizer solves the equation
            \begin{align*}
                \begin{cases}
                    \Delta u_0=0 &\text{in }M\\
                    \frac{\partial u_0}{\partial \nu}+\lambda u_0=Q_{q,\lambda}(u_0)u_0^q
                     &\text{on }\partial M.
                \end{cases}
            \end{align*}
            Let $u:=Q_{q,\lambda}(u_0)^{\frac{1}{q-1}}u_0 $, then $u$ solves
            \begin{align*}
                \begin{cases}
                    \Delta u=0&\text{in }M\\
                    \frac{\partial u}{\partial \nu}+\lambda u=u^q
                     &\text{on }\partial M.
                \end{cases}
            \end{align*}
            Since $u_0$ is a minimizer, by competing with the constant function,  we have
            \begin{align*}
                Q_{q,\lambda}(u_0)\leq
                \lambda |\partial M|^{\frac{q-1}{q+1}}.
            \end{align*}
            Combining with $1=\int_{\partial M}u_0^{q+1}            =\int_{\partial\mathbb{B}^n}Q_{q,\lambda}(u_0)^{-\frac{q+1}{q-1}}u^{q+1} $, we derive that
            \begin{align*}
                \int_{\partial M}u^{q+1}
                =Q_{q,\lambda}(u_0)^{\frac{q+1}{q-1}}
                \leq \lambda^{\frac{q+1}{q-1}}
                |\partial M|.
            \end{align*}
            \textbf{Step 2:} Define $d:=\frac{2q}{q-1}$ and let $v:=u^{-\frac{2}{d-2}}$, $\chi:=\frac{\partial v}{\partial \nu}|_{\partial M}, f:=v|_{\partial M} $, then $v$ solves \cref{eq. PDE of v on warped}
        with 
        \begin{align}\label{ineq. integral hypothesis}
            \int_{\partial M}f^{1-d}
            \leq \lambda^{d-1}|\partial M|.
        \end{align}
        As before (see  \cref{eq. warped 1}), there holds the key integral identity:
        \begin{align}\label{eq. 1 in warped case}
            &\int_{M}v^{1-d}\left(
        \left|\nabla^2 v-\frac{\Delta v}{n}g\right|^2
        +\left(\frac{1}{n}-\frac{1}{d}\right)(\Delta v)^2
        +Ric(\nabla v,\nabla v)
        \right)\notag\\
        =&\left(\frac{2}{d-2}\lambda
        -\frac{\rho'(R)}{\rho(R)}\right)
        \int_{\partial M}f^{1-d}|\nabla f|^2
        +\frac{2}{d-2}\lambda
        \int_{\partial M}f^{2-d}
        \left(\nabla^2 v-\frac{\Delta v}{d}g\right)(\nu,\nu)\notag\\
        &-\frac{2}{d-2}
        \int_{\partial M}f^{1-d}
        \left(\nabla^2 v-\frac{\Delta v}{d}g\right)(\nu,\nu).
        \end{align}
        \textbf{Step 3:}
        We tackle the second term in the right hand side of \cref{eq. 1 in warped case}. By \cref{eq. PDE of v on warped} and divergence theorem we have
        \begin{align}
            &\int_{\partial M}f^{2-d}\left(\nabla^2 v-\frac{\Delta v}{d}g\right)(\nu,\nu)
            =\int_{\partial M}f^{2-d}\left(\nabla^2 v(\nu,\nu)-\frac{1}{d}\Delta v\right) \notag\\
            =&\int_{\partial M}
            f^{2-d}\left(\frac{d-1}{d}\Delta v-\Delta f-(n-1)\frac{\rho'(R)}{\rho(R)}\chi\right)\notag\\
            =&\int_{\partial M}
            f^{1-d}\left(\frac{d-1}{2}|\nabla f|^2-f\Delta f+\frac{d-1}{2}\chi^2
            -(n-1)\frac{\rho'(R)}{\rho(R)}f\chi\right)\notag\\
            =&\int_{\partial M}
            f^{1-d}\left(\frac{d-1}{2}|\nabla f|^2-f\Delta f
            -\frac{\lambda}{d-2}(n-1)\frac{\rho'(R)}{\rho(R)}f^2\right)\notag\\
            &+\frac{(d-2)(n-1)}{4\lambda}\frac{\rho'(R)}{\rho(R)}
            \int_{\partial M}
            f^{1-d}\left(\chi-\frac{2}{d-2}\lambda f\right)^2
            +
            \left(\frac{d-1}{2}-\frac{(d-2)(n-1)}{4\lambda}\frac{\rho'(R)}{\rho(R)}\right)
            \int_{\partial M}
            f^{1-d}\chi^2\notag\\
            =&\int_{\partial M}
            f^{1-d}\left(-\frac{d-3}{2}|\nabla f|^2
            -\frac{(n-1)\lambda}{d-2}\frac{\rho'(R)}{\rho(R)}
            f^2\right)
            +\frac{n-1}{(d-2)\lambda}
            \frac{\rho'(R)}{\rho(R)}
            \int_{\partial M} f^{1-d}\notag\\
            &+\left(\frac{d-1}{2}-\frac{(d-2)(n-1)}{4\lambda}\frac{\rho'(R)}{\rho(R)}\right)
            \int_{\partial M}
            f^{1-d}\chi^2\notag\\
            \leq& \int_{\partial M}
            f^{1-d}\left(-\frac{d-3}{2}|\nabla f|^2
            -\frac{(n-1)\lambda}{d-2}\frac{\rho'(R)}{\rho(R)}
            f^2\right)
            +\frac{n-1}{(d-2)\lambda}
            \frac{\rho'(R)}{\rho(R)}
            \int_{\partial M} f^{1-d},\label{ineq. 2-d in warped case}
        \end{align}
        where we used $0<\lambda\leq
        \frac{d-2}{2}\frac{n-1}{d-1}\frac{\rho'(R)}{\rho(R)}
        =\frac{n-1}{q+1}\frac{\rho'(R)}{\rho(R)}$ in the last line.
        
        If $d=3$, we could utilize \cref{ineq. integral hypothesis} to  continue \cref{ineq. 2-d in warped case} as follows
        \begin{align*}
            &\int_{\partial M}
            f^{1-d}\left(-\frac{d-3}{2}|\nabla f|^2
            -\frac{(n-1)\lambda}{d-2}\frac{\rho'(R)}{\rho(R)}
            f^2\right)
            +\frac{n-1}{(d-2)\lambda}
            \frac{\rho'(R)}{\rho(R)}
            \int_{\partial M} f^{1-d}\\
            =&-(n-1)\lambda\frac{\rho'(R)}{\rho(R)}
            |\partial M|
            +\frac{n-1}{\lambda}\frac{\rho'(R)}{\rho(R)}\int_{\partial M}f^{1-d}\\
            \leq &-(n-1)\lambda\frac{\rho'(R)}{\rho(R)}
            |\partial M|
            +(n-1)\lambda^{d-2}\frac{\rho'(R)}{\rho(R)}|\partial M|=0.
        \end{align*}
        If $d>3$, let $\tilde{f}:=f^{\frac{3-d}{2}} $, then we continue \cref{ineq. 2-d in warped case} as follows
        \begin{align*}
            &\int_{\partial M}
            f^{1-d}\left(-\frac{d-3}{2}|\nabla f|^2
            -\frac{(n-1)\lambda}{d-2}\frac{\rho'(R)}{\rho(R)}f^2\right)
            +\frac{n-1}{(d-2)\lambda}
            \frac{\rho'(R)}{\rho(R)}
            \int_{\partial M} f^{1-d}\\
            =&-\frac{2}{d-3}
            \int_{\partial M}
            \left(|\nabla \tilde{f}|^2
            +\frac{\lambda(d-3)(n-1)}{2(d-2)}\frac{\rho'(R)}{\rho(R)}
            \tilde{f}^2\right)
            +\frac{n-1}{(d-2)\lambda}
            \frac{\rho'(R)}{\rho(R)}
            \int_{\partial M} f^{1-d}\\
            \leq &
            -\frac{\lambda(n-1)}{d-2}
            \frac{\rho'(R)}{\rho(R)}
            \left(\int_{\partial M}f^{1-d}\right)^{\frac{d-3}{d-1}}
            |\partial M|^{\frac{2}{d-1}}
            +\frac{n-1}{(d-2)\lambda}
            \frac{\rho'(R)}{\rho(R)}
            \int_{\partial M} f^{1-d}\\
            =&\frac{n-1}{d-2}\frac{\rho'(R)}{\rho(R)}
            \left(\int_{\partial M}f^{1-d}\right)^{\frac{d-3}{d-1}}
            \left(
            -\lambda|\partial M|^{\frac{2}{d-1}}
            +\frac{1}{\lambda}
            \left(\int_{\partial M}
            f^{1-d}\right)^{\frac{2}{d-1}}
            \right)\leq 0,
        \end{align*}
        where in the penultimate line,  we used that $\partial M$ is a round sphere  and the sharp Sobolev inequality on $\mathbb{S}^{n-1}(\rho(R))$ (see \cite[Corollary 6.2]{BVV91}):
        \begin{align*}
            &\int_{\mathbb{S}^{n-1}(\rho(R))}|\nabla u|^2+su^2
            \geq 
            s\left(
            \int_{\mathbb{S}^{n-1}(\rho(R))}|u|^{p+1}
            \right)^{\frac{2}{p+1}}|\mathbb{S}^{n-1}(\rho(R))|^{\frac{p-1}{p+1}},\quad \forall u\in C^{\infty}(\mathbb{S}^{n-1}(\rho(R))),
        \end{align*}
        where $1<p\leq\frac{n+1}{n-3}\mathrm{~and~}0\leq s\leq \frac{n-1}{p-1}$. Here we chose $p=\frac{d+1}{d-3}\leq \frac{n+1}{n-3}$, and $s=\frac{\lambda(d-3)(n-1)}{2(d-2)}\frac{\rho'(R)}{\rho(R)}\leq \frac{(d-3)(n-1)}{4}=\frac{n-1}{p-1}$. Note that we used $\lambda\leq \frac{1}{q-1}\frac{\rho(R)}{\rho'(R)} $. 
        
        In conclusion, we obtain that
        \begin{align*}
            \int_{\partial M}f^{2-d}
            \left(\nabla^2 v-\frac{\Delta v}{d}g\right)(\nu,\nu)
            \leq 0.
        \end{align*}
        \textbf{Step 4:}  We tackle the last term in the right hand side of \cref{eq. 1 in warped case}. Consider the weight function \cref{eq. weight in warped case}.
        Then $w$ satisfies: $w|_{\partial M}=0,\ \nabla w|_{\partial M}=-\rho(R)\nu,\ \Delta w=-n\rho'$, and $\nabla^2 w=\frac{\Delta w}{n}g$.
        
        As in \cref{sec3}, there holds the Pohozaev identity (see \cref{eq. warped 1w_1-d}):
        \begin{align*}
            \rho(R)\int_{\partial M}
            f^{1-d}\left(\nabla^2v-\frac{\Delta v}{d}g\right)(\nu,\nu)
            =\left(1-\frac{n}{d}\right)\int_M 
            \rho'v^{1-d}\Delta v
            -(n-1)\int_M \rho''v^{1-d}\frac{\partial v}{\partial r}.
        \end{align*}
        Combining with \cref{eq. integral trick}, the Pohozaev identity turns out to be
        \begin{align*}
             &\rho(R)\int_{\partial M}
            f^{1-d}
            \left(\nabla^2v-\frac{\Delta v}{d}g\right)
            (\nu,\nu)\notag\\
            =&\int_M \left(-\frac{(n-2)(d-1)}{d}\rho'
            +\frac{(n-1)(d-2)}{d}\rho'(R)\right)
            v^{1-d}\Delta v\\
            \geq&\int_M
            \left(
            -\frac{(n-2)(d-1)}{d}
             +\frac{(n-1)(d-2)}{d}\rho'(R)
            \right)v^{1-d}\Delta v
            \geq 0,
        \end{align*}
where in the last line, we used $\rho'(r)\leq 1$ (see \cref{ineq. rho prime r}), and
 the last inequality is guaranteed by
            \begin{align*}
                d\geq \frac{2(n-1)\rho'(R)-(n-2)}{(n-1)\rho'(R)-(n-2)},
            \end{align*}
             which is exactly our hypothesis on $q$:
            $1<q\leq \frac{2(n-1)\rho'(R)-(n-2)}{n-2}$.

        \textbf{Step 5:} Note that there always holds $\rho'(R)\leq 1$ (see \cref{ineq. rho prime r}).
        
        If $\rho'(R)<1$, then $q\leq \frac{2(n-1)\rho'(R)-(n-2)}{n-2}<\frac{n}{n-2}$. It follows that $d:=\frac{2q}{q-1}>n$. 
        Then by step 3, step 4 and the key integral identity \cref{eq. 1 in warped case}, we get
        \begin{align*}
            0\leq &\int_{M}v^{1-d}\left(
        \left|\nabla^2 v-\frac{\Delta v}{n}g\right|^2
        +\left(\frac{1}{n}-\frac{1}{d}\right)(\Delta v)^2
        +Ric(\nabla v,\nabla v)
        \right)\notag\\
        =&\left(\frac{2}{d-2}\lambda
        -\frac{\rho'(R)}{\rho(R)}\right)
        \int_{\partial M}f^{1-d}|\nabla f|^2
        +\frac{2}{d-2}\lambda
        \int_{\partial M}f^{2-d}
        \left(\nabla^2 v-\frac{\Delta v}{d}g\right)(\nu,\nu)\notag\\
        &-\frac{2}{d-2}
        \int_{\partial M}f^{1-d}
        \left(\nabla^2 v-\frac{\Delta v}{d}g\right)(\nu,\nu)
        \leq 0,
        \end{align*}
        which forces $(\frac{d}{2}v^{-1}|\nabla v|^2)^2=(\Delta v)^2\equiv 0$ since $d>n$. Therefore $v$ is a constant function and so is $u_0$.

        If $\rho'(R)=1$, then $\rho'\equiv 1$ by \cref{ineq. rho prime r}. We get
            $\rho(r)=r,\quad \forall r\in[0,R]$.
        In other words, $(M^n,g)$ is isometric to a round ball of radius $R$. As before, if $R<1$, then step 3, step 4 and the identity \cref{eq. 1 in warped case} yield 
        \begin{align*}
            \int_{\partial M}f^{1-d}|\nabla f|^2=0,
        \end{align*}
        so $v|_{\partial M}=f$ is a constant function. Therefore, $u$ is a harmonic function on $(M^n,g)$ with a constant boundary value. It follows that $u$ is a constant and so is $u_0$.

        Finally, the remaining case is that $(M^n,g)$ is isometric to $\mathbb{B}^n$, on which the solution $u$ is completely classified by \cref{thm. GL25}.
\end{proof}
\textbf{Acknowledgements} This work was completed during a research visit to Michigan State University
. I wish to thank the university for its hospitality and am especially grateful to Prof. Xiaodong Wang and Dr. Zhixin Wang for their helpful discussions. I would also like to thank Prof. Abdolhakim Shouman for drawing my attention to the work \cite{Sho26}.

\textbf{Funding} This work is partially supported by NSFC-12031012, NSFC-12171313 and NSFC-12250710674.

 \textbf{Data Availability} Data sharing not applicable to this article as no datasets were generated or analyzed during
 the current study.
 
 \textbf{Conflicts of Interest}  The author declares that there are no conflicts of interest.
\bibliographystyle{amsplain}
\bibliography{bib}
\end{document}